\documentclass[11pt]{amsproc}
\usepackage[utf8]{inputenc}
\usepackage{amsthm}
\usepackage{amsmath}
\usepackage{ragged2e}
\usepackage{amssymb}
\usepackage[margin=1.5in]{geometry}
\usepackage[dvipsnames]{xcolor}
\usepackage{tikz}
\usetikzlibrary{hobby}
\usepackage{color}
\usepackage[colorlinks=true, pdfstartview=FitV, linkcolor=blue, citecolor=blue, urlcolor=blue]{hyperref}
\usepackage{comment}

\newcommand\N{\mathbb{N}}

\newtheorem{Thm}{Theorem}
\newtheorem{Lemma}[Thm]{Lemma}

\newtheorem{Cor}[Thm]{Corollary}

\newtheorem{Prop}[Thm]{Proposition}

\theoremstyle{remark}
\newtheorem{Rem}[Thm]{Remark}

\numberwithin{equation}{section}
\usepackage{mathtools}
\mathtoolsset{showonlyrefs}

\title[Rigidity in the Planar Ulam Floating Body Problem]{Rigidity in the Planar Ulam Floating Body Problem with Perimetral Density ~$\sigma=\tfrac16$}
\author{Oleg Asipchuk}
\address{Oleg Asipchuk,  University of Cincinnati,
	Department of Mathematical Sciences,
	Cincinnati, OH 45221, USA}
\email{asipchah@ucmail.uc.edu}

\author{Maksim Kosmakov}
\address{Maksim Kosmakov,  University of Cincinnati,
	Department of Mathematical Sciences,
	Cincinnati, OH 45221, USA}
\email{kosmakmm@ucmail.uc.edu}

\author{Pavel Zatitskii}
\address{Pavel Zatitskii,  University of Cincinnati,
	Department of Mathematical Sciences,
	Cincinnati, OH 45221, USA}
\email{zatitspl@ucmail.uc.edu}


\begin{document}
 
\begin{abstract}
We study the two-dimensional Ulam's floating body problem for convex domains with perimetral density $\sigma=\tfrac16$. Using the framework of Zindler carousels, we reduce the problem to a two-dimensional dynamical system associated with an inscribed equilateral hexagon. Our main result shows that the disk is the only convex domain floating in equilibrium in every position for this perimetral density. This  provides a new rigidity result for rational perimetral densities in the convex setting.
\end{abstract}

\maketitle

\section*{Introduction}

\begin{center}
    {\it Is a solid of uniform density which will float in water in every position a sphere?}
\end{center}
\begin{FlushRight}
The Scottish Book \cite{ScottishBook2015}, Problem 19
\end{FlushRight}

The two-dimensional version of this question concerns a cylinder of uniform density floating in every orientation with its axis parallel to the water surface. By Archimedes' law, the problem reduces to the geometry of the planar cross-section. In a classical paper \cite{Auerbach1938}, Auerbach showed that for density $\tfrac12$ the cross-section need not be a disk, and in fact need not even be convex. The corresponding planar curves are precisely the Zindler curves introduced earlier in \cite{Zindler1921}. Later, Montejano~\cite{Montejano1974} proved that for density $0$ the only solution is the disk.

If a planar domain $D$ of uniform density floats in equilibrium in every orientation, then each waterline cuts the boundary $\partial D$ into two arcs whose lengths are in a fixed ratio
$\sigma:(1-\sigma)$
The parameter $\sigma$ is called the \emph{perimetral density}. Thus the floating-body problem may be viewed as the problem of describing plane convex bodies with a prescribed perimetral density that float in equilibrium in every position.

The case $\sigma=\tfrac12$ is exceptional because it admits a large family of noncircular examples. For other values of the density, the picture is much more rigid. Bracho, Montejano, and Oliveros \cite{BrachoMontejanoOliveros2001} proved that for perimetral densities $\sigma=\tfrac13$ and $\sigma=\tfrac14$ the disk is the only convex floating body. They also obtained nonexistence results for $\sigma=\tfrac15$ and $\sigma=\tfrac25$ providing numerical evidence. On the other hand, Wegner~\cite{Wegner2003,Wegner2007} constructed noncircular examples for some other densities by perturbing the circular solution.

The problem of Ulam has also been studied in higher dimensions. We refer the reader to surveys \cite{RyaboginSurvey,ScottishBook2015}, as well as to \cite{Falconer1983,Ryabogin2022,Ryabogin2023,Schneider1970} and references therein. We also mention the classical works \cite{Poussin1925,Dupin1822,Zhukovsky1936}, which contain several foundational ideas in the subject.

In this paper we consider the next unresolved rational case in the convex setting, namely, the perimetral density $\sigma=\frac16$.
Our main result is the following.

\begin{Thm}\label{thm:main}
Let $K\subset \mathbb{R}^2$ be a strictly convex body with $C^1$ boundary. If $K$ floats in equilibrium in every position with perimetral density $1/6$, then $K$ is a disk.
\end{Thm}
Our approach uses the framework of Bracho, Montejano, and Oliveros and reformulates the problem in terms of a Zindler  carousel. The carousel structure forces central symmetry and reduces the problem to a two-dimensional Hamiltonian system for two angle variables, whose first integral is the area of the inscribed hexagon. The proof of Theorem~\ref{thm:main} then follows from analysis of the corresponding dynamical system together with geometric constraints coming from convexity.

The paper is organized as follows. In Section~\ref{sec:Geom_str}, we discuss properties of floating bodies and Zindler carousels and fix the notations. In Section~\ref{sec:Hex}, we derive several auxiliary geometric relations and reduce the problem to a planar system of ODEs. Section~\ref{sec:Main_th} contains the proof of the main theorem. In Section~\ref{sec:rem}, we collect several concluding remarks. Appendix~\ref{app:bounds} contains the proofs of the technical estimates used in the main argument.

\medskip
\noindent
{\it Acknowledgments}.
The authors would like to thank Dmitry Ryabogin for pointing our attention to the question, for many insightful mathematical discussions, and for his valuable comments and suggestions that significantly improved this article.

\section{Geometric structure of Zindler carousels} 
\label{sec:Geom_str}
In this section, we introduce the necessary notation for our problem and present several results that will be used in the proof. 

We first recall the equivalence between convex floating bodies of rational perimetral density and Zindler carousels; see \cite{BrachoMontejanoOliveros2001,BrachoMontejanoOliveros2004}.
In the special case relevant for us, this equivalence takes the following form.

\begin{Thm}\label{thm:Zindler_car}
Consider a strictly convex body $K$. Let $P$ be its perimeter, and let $\gamma:\mathbb{R}/P\mathbb{Z}\to\mathbb{R}^2$ be a $C^1$-smooth parametrization of its boundary by arc length. Then the following properties are equivalent:

\begin{enumerate}
\item $K$ floats in equilibrium in every position with perimetral density $\sigma=\frac{1}{N}$.

\item There exists a constant $\ell$, $\ell>0$, independent of $t$, such that for every $t\in \mathbb{R}/P\mathbb{Z}$, the points
\begin{equation*}
v_i(t):=\gamma\!\left(t+\frac{(i-1)P}{N}\right),
\qquad i=1,\dots,N,
\end{equation*}
are the vertices of an inscribed equilateral $N$-gon  with side length $\ell$:
\begin{equation*}
|v_{i+1}(t)-v_i(t)|=\ell,
\qquad i=1,\dots,N,
\end{equation*}
where indices are understood modulo $N$.
\end{enumerate}
\end{Thm}

\begin{Rem}
For each $t$, the points $v_1(t),\dots,v_{N}(t)$ lie on the strictly convex curve $\gamma$ in cyclic order, and therefore form a convex inscribed $N$-gon. Moreover, the condition that the side lengths are independent of $t$ implies that, as $t$ varies, the midpoint $m_i(t):=(v_i(t)+v_{i+1}(t))/2$
of each side moves in a direction parallel to that side:
$\dot{m_i}(t)\parallel (v_{i+1}(t)-v_i(t))$.
\end{Rem}

Without loss of generality, we can assume that the side of the $N$-gon is $\ell=2$.
We denote by $x_i(t)$ the interior angle of the inscribed $N$-gon at the vertex $v_i(t)$. The carousel condition implies that the tangent lines to $\gamma$ at the two consecutive vertices $v_i(t)$ and $v_{i+1}(t)$ make equal angles with the side
$[v_i(t),v_{i+1}(t)]$.  We denote those angles as~$\alpha_i(t)$, see Figure~\ref{F-1}.

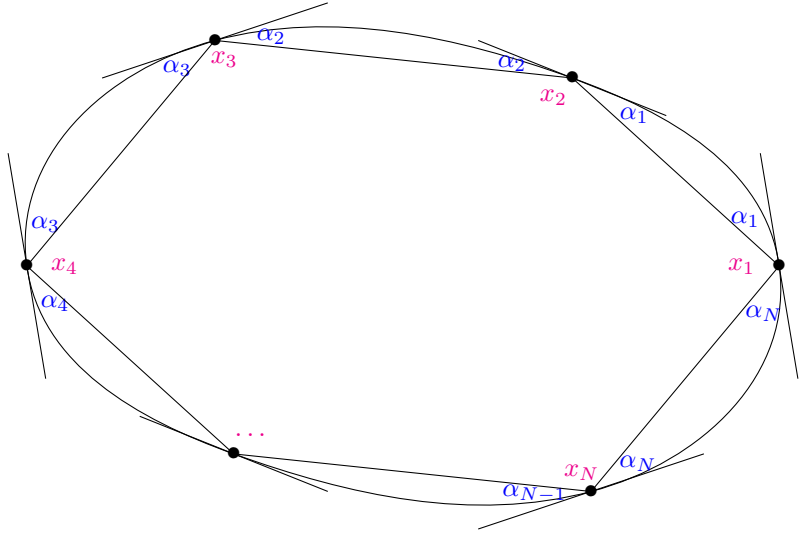
\begin{figure}
    \centering
    \begin{tikzpicture}[scale=2.5]
\coordinate (a) at (2,0);
\coordinate (b) at (0.9,1);
\coordinate (c) at (-1,1.2);
\coordinate (d) at (-2,0);
\coordinate (e) at (-0.9,-1);
\coordinate (f) at (1,-1.2);

\path[draw,use Hobby shortcut,closed=true]
(a)..(b)..(c)..(d)..(e)..(f);

\draw (a) node{$\bullet$};
\draw (b) node{$\bullet$};
\draw (c) node{$\bullet$};
\draw (d) node{$\bullet$};
\draw (e) node{$\bullet$};
\draw (f) node{$\bullet$};

\draw    (1.8,0) node [magenta]{{\small $x_1$}};
\draw    (0.8,0.9) node [magenta]{{\small $x_2$}};
\draw    (-0.95,1.1) node [magenta]{{\small $x_3$}};
\draw    (-1.8,0) node [magenta]{{\small $x_4$}};
\draw    (-0.8,-0.9) node [magenta]{{\small $\dots$}};
\draw    (0.95,-1.1) node [magenta]{{\small $x_N$}};

\draw    (1.82,0.25) node [blue]{{\small $\alpha_1$}};
\draw    (1.92,-0.25) node [blue]{{\small $\alpha_{N}$}};

\draw    (1.23,0.8) node [blue]{{\small $\alpha_1$}};
\draw    (0.58,1.08) node [blue]{{\small $\alpha_2$}};

\draw    (-0.7,1.22) node [blue]{{\small $\alpha_2$}};
\draw    (-1.2,1.05) node [blue]{{\small $\alpha_3$}};

\draw    (-1.85,-0.2) node [blue]{{\small $\alpha_4$}};
\draw    (-1.9,0.22) node [blue]{{\small $\alpha_3$}};

\draw    (0.7,-1.21) node [blue]{{\small $\alpha_{N-1}$}};
\draw    (1.25,-1.05) node [blue]{{\small $\alpha_{N}$}};

\draw (a)--(b)--(c)--(d)--(e)--(f)--(a);

\draw [black] (2,0) --(2.1,-0.6);
\draw [black] (2,0) --(1.9,0.6);

\draw [black] (0.9,1) --(0.4,1.2);
\draw [black] (0.9,1) --(1.4,0.8);

\draw [black] (-1,1.2) --(-0.4,1.4);
\draw [black] (-1,1.2) --(-1.6,1);

\draw [black] (-2,0) --(-2.1,0.6);
\draw [black] (-2,0) --(-1.9,-0.6);

\draw [black] (-0.9,-1) --(-0.4,-1.2);
\draw [black] (-0.9,-1) --(-1.4,-0.8);

\draw [black] (1,-1.2) --(0.4,-1.4);
\draw [black] (1,-1.2) --(1.6,-1);

\end{tikzpicture}
    \caption{A convex body $K$ with an inscribed $N$-gon.}
    \label{F-1}
\end{figure}

The following theorem is an adaptation of Theorem~6 in \cite{BrachoMontejanoOliveros2004}. 

\begin{Thm}\label{T-systemODE}
   Under the conditions of Theorem~\ref{thm:Zindler_car} with side length $\ell = 2$, the interior angles $x_i(t)$, $i = 1, \dots, N$, satisfy the following system of differential equations:
\begin{equation}\label{E-systemODE}
\dot{x}_i(t) = \sin(\alpha_{i-1}(t)) - \sin(\alpha_i(t)),
\end{equation}
where $\dot{x}=\dfrac{dx}{dt}$ denotes differentiation with respect to the arc-length parameter $t$.
\end{Thm}

Now, we consider the case when $N$ is even: $N=2n$ for some $n\in\mathbb{N}$. We begin with a simple geometric observation. 
If for each $t$ the inscribed $2n$-gon $V(t)$ is centrally symmetric, then its center of symmetry is a priori allowed to depend on~$t$. 
The next lemma shows that the carousel midpoint condition forces this center to be constant, and hence the boundary curve itself is centrally symmetric.
 \begin{Lemma}\label{lem:constant_center} 
Let $\gamma:\mathbb{R}/P\mathbb{Z}\to\mathbb{R}^2$ be an arc-length parametrization of the boundary of a floating body with perimetral density $\sigma=\tfrac{1}{2n}$. Set
\begin{equation}
\mu:=\frac{P}{2n},
\qquad
v_i(t):=\gamma\bigl(t+(i-1)\mu\bigr),
\qquad i=1,\dots,2n.
\end{equation}
Assume that for every $t$, the inscribed $2n$-gon
$V(t)$ with the vertices $v_1(t),\dots,v_{2n}(t)$ is centrally symmetric. Then the boundary curve $\gamma$ is centrally symmetric, i.\,e., there exists a fixed point $c\in\mathbb{R}^2$ such that
\begin{equation}
\gamma\!\left(t+\frac{P}{2}\right)=2c-\gamma(t)
\qquad\text{for all }t\in \mathbb{R}/P\mathbb{Z}.
\end{equation}
\end{Lemma}

\begin{proof}
For each $t$, let $c(t)$ be the center of symmetry of the inscribed centrally symmetric $2n$-gon $V(t)$. Thus
\begin{equation}\label{eq:opp_vertices_short}
v_{i+n}(t)=2c(t)-v_i(t),
\qquad i=1,\dots,n,
\end{equation}
with indices understood modulo $2n$.
Let
\begin{equation}
e_i(t):=v_{i+1}(t)-v_i(t),
\qquad
m_i(t):=\frac{v_i(t)+v_{i+1}(t)}2
\end{equation}
be the side vectors and the side midpoints. From \eqref{eq:opp_vertices_short} we  get
\begin{equation}
e_{i+n}(t)=-e_i(t),
\qquad
m_{i+n}(t)=2c(t)-m_i(t),
\qquad i=1,\dots,n.
\end{equation}
We differentiate the second relation and get
\begin{equation}
2\dot c(t)=\dot m_i(t)+\dot m_{i+n}(t).
\end{equation}
Now we use the carousel midpoint condition: $\dot m_i(t)$ is parallel to $e_i(t)$ for every~$i$. Since $e_{i+n}(t)=-e_i(t)$, the vector $\dot m_{i+n}(t)$ is also parallel to $e_i(t)$. Thus,
$2\dot c(t)=\dot m_i(t)+\dot m_{i+n}(t)$ is parallel to $e_i(t)$ for every $i=1,\dots,n$.
Since $V(t)$ is a nondegenerate convex polygon, not all its side vectors $e_i(t)$ are parallel. Therefore the only vector parallel to all side vectors is the zero vector. This means $\dot c(t)=0$, and thus $c(t)\equiv c$ is constant. Taking $i=1$ in \eqref{eq:opp_vertices_short}, we get
\begin{equation}
\gamma(t+n\mu)=2c-\gamma(t).
\end{equation}
Since $\mu=\frac{P}{2n}$, we have $n\mu=\frac P2$, and thus
$\gamma\!\left(t+\frac P2\right)=2c-\gamma(t)$, which means $\gamma$ is centrally symmetric.
\end{proof}

\section{Properties of the inscribed hexagon}\label{sec:Hex}

We now specialize to the case $2n=6$, corresponding to the inscribed
equilateral hexagon for $\sigma=\tfrac16$. We first show that such a hexagon is centrally symmetric. This reduces the carousel dynamics to a two-dimensional Hamiltonian system, whose period estimate is then used in the proof of Theorem~\ref{thm:main}.

\subsection{Central symmetry}

    \begin{figure}[h]
    \centering
    \begin{tikzpicture}[scale=2.5]
\coordinate (a) at (2,0);
\coordinate (b) at (0.9,1);
\coordinate (c) at (-1,1.2);
\coordinate (d) at (-1.9,0);
\coordinate (e) at (-0.9,-1.1);
\coordinate (f) at (1,-1.4);

\draw (a) node{$\bullet$};
\draw (b) node{$\bullet$};
\draw (c) node{$\bullet$};
\draw (d) node{$\bullet$};
\draw (e) node{$\bullet$};
\draw (f) node{$\bullet$};

\draw    (1.8,0) node [magenta]{{\small $x_1$}};
\draw    (-0.95,1.1) node [magenta]{{\small $x_3$}};
\draw    (-0.8,-1) node [magenta]{{\small $x_5$}};

\draw    (0.8,0.9) node [blue]{{\small $\psi_1$}};
\draw    (-1.7,0) node [blue]{{\small $\psi_2$}};
\draw    (0.9,-1.2) node [blue]{{\small $\psi_3$}};

\draw    (1.45,0.6) node [magenta]{{\small $2$}};
\draw    (0,1.2) node [magenta]{{\small $2$}};
\draw    (-1.6,0.6) node [magenta]{{\small $2$}};
\draw    (-1.45,-0.7) node [magenta]{{\small $2$}};
\draw    (0,-1.35) node [magenta]{{\small $2$}};
\draw    (1.6,-0.7) node [magenta]{{\small $2$}};

\draw    (-0.6,-0.7) node [blue]{{\small $l_1$}};
\draw    (1.05,0) node [blue]{{\small $l_2$}};
\draw    (-0.6,0.65) node [blue]{{\small $l_3$}};

\draw (a)--(b)--(c)--(d)--(e)--(f)--(a);

\draw (b)--(d);
\draw (d)--(f);
\draw (f)--(b);

\end{tikzpicture}
    \caption{Convex hexagon $\mathcal{H}$.}
    \label{F-hexagon}
\end{figure}
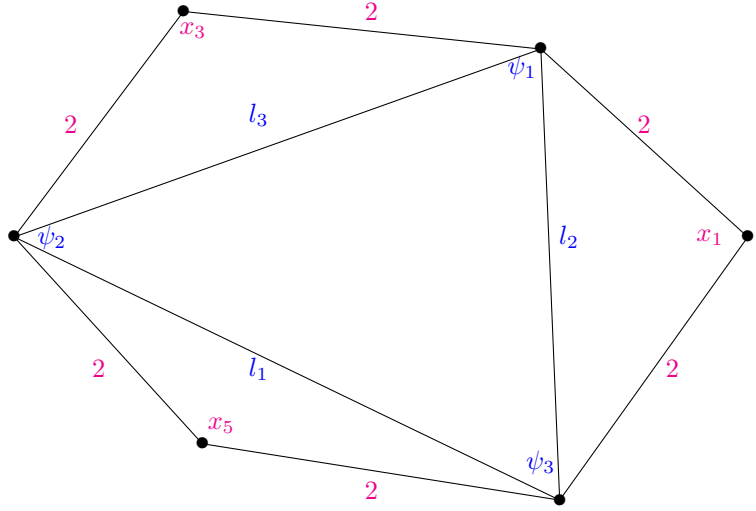

\begin{Lemma} \label{lem:sym_Hex}
Let $\mathcal{H}$ be a convex equilateral hexagon with interior angles
$x_1, \dots, x_6 \in\left(\frac{\pi}{2},\pi\right)$.
Assume that
\begin{equation}\label{eq:odd_even_sums_hexagon}
x_1+x_3+x_5=2\pi
\qquad
\text{and}
\qquad
x_2+x_4+x_6=2\pi.
\end{equation}

Then  $\mathcal{H}$ is centrally symmetric.
\end{Lemma}

\begin{proof}
Set $a=\frac{x_1}{2}, b=\frac{x_3}{2}$, and $c=\frac{x_5}{2}$.
 By \eqref{eq:odd_even_sums_hexagon} we have $a+b+c=\pi$.
One can show that  the angles of the interior triangle from Figure~\ref{F-hexagon} are
\begin{equation}\label{eq:psi_hexagon}
\psi_1=x_2-\frac{x_5}{2},
\qquad
\psi_2=x_4-\frac{x_1}{2},
\qquad
\psi_3=x_6-\frac{x_3}{2},
\end{equation}
and the corresponding side lengths are 
\begin{equation}
    l_1=4 \sin c,\quad  l_2=4 \sin a, \quad  l_3=4 \sin b.
\end{equation}
Using the law of sines, we obtain
\begin{equation}\label{eq:sine_ratio_hexagon}
\frac{\sin\psi_1}{\sin c}
=
\frac{\sin\psi_2}{\sin a}
=
\frac{\sin\psi_3}{\sin b}.
\end{equation}
These identities together with the fact that $\psi_1+\psi_2+\psi_3=\pi=c+a+b$ give $\psi_1=c$, $\psi_2=a$, and~$\psi_3=b$. Indeed, one can construct a triangle with the angles $a$, $b$, and $c$, and note that it is similar to the triangle with the angles $\psi_1$, $\psi_2$, and $\psi_3$ because of~\eqref{eq:sine_ratio_hexagon}.
It follows that $x_2=x_5$, $x_4=x_1$, and $x_6=x_3$. For a convex hexagon, equality of opposite interior angles implies that the opposite sides are parallel. Since the hexagon is equilateral, opposite sides are also equal in length, so the three pairs of opposite vertices have the same midpoint. Therefore $\mathcal{H}$ is centrally symmetric.
\end{proof}

Combining Lemma~\ref{lem:constant_center} with Lemma~\ref{lem:sym_Hex}, we obtain the following corollary.
\begin{Cor}\label{L-center}
    If a convex body $K$ satisfies the conditions of Theorem~\ref{thm:Zindler_car} with density $\sigma=\frac{1}{6}$, then $K$ must be centrally symmetric.
\end{Cor}

\begin{proof}
    One has  $\sum_{i=1}^6 x_i=4\pi$ and 
        $x_i=\pi-\alpha_i-\alpha_{i-1}$, for $i=1,\ldots 6$. Thus,
$x_1+x_3+x_5=x_2+x_4+x_6=2\pi$.
Therefore we can apply Lemma~\ref{lem:sym_Hex} and then Lemma~\ref{lem:constant_center}.
\end{proof}

Under the central symmetry (see Figure~\ref{F-3}), the general system \eqref{E-systemODE} reduces to 
\begin{equation}\label{E-3by3}
    \begin{cases}
       \dot{x}_1(t) = \sin(\alpha_{3}(t)) - \sin(\alpha_1(t)),\\
        \dot{x}_2(t) = \sin(\alpha_{1}(t)) - \sin(\alpha_2(t)),\\
        \dot{x}_3(t) = \sin(\alpha_{2}(t)) - \sin(\alpha_3(t)).
    \end{cases}
\end{equation}

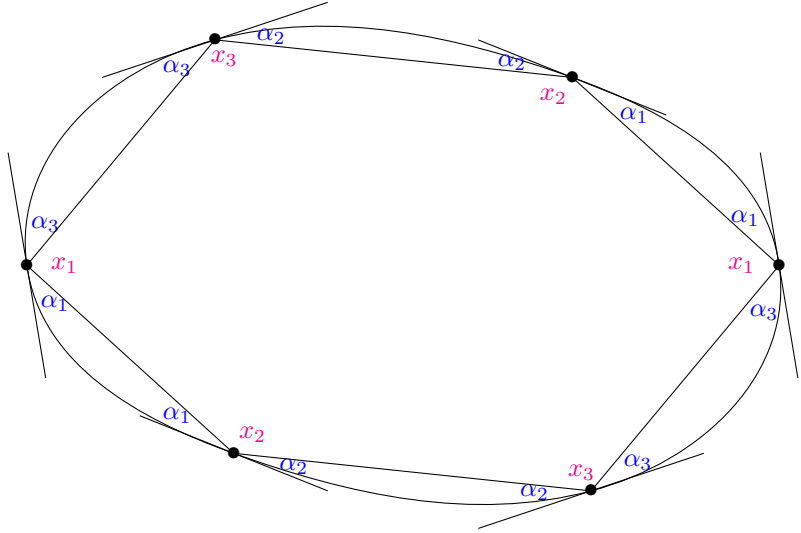
\begin{figure}[h]
    \centering
    \begin{tikzpicture}[scale=2.5]
\coordinate (a) at (2,0);
\coordinate (b) at (0.9,1);
\coordinate (c) at (-1,1.2);
\coordinate (d) at (-2,0);
\coordinate (e) at (-0.9,-1);
\coordinate (f) at (1,-1.2);

\path[draw,use Hobby shortcut,closed=true]
(a)..(b)..(c)..(d)..(e)..(f);

\draw (a) node{$\bullet$};
\draw (b) node{$\bullet$};
\draw (c) node{$\bullet$};
\draw (d) node{$\bullet$};
\draw (e) node{$\bullet$};
\draw (f) node{$\bullet$};

\draw    (1.8,0) node [magenta]{{\small $x_1$}};
\draw    (0.8,0.9) node [magenta]{{\small $x_2$}};
\draw    (-0.95,1.1) node [magenta]{{\small $x_3$}};
\draw    (-1.8,0) node [magenta]{{\small $x_1$}};
\draw    (-0.8,-0.9) node [magenta]{{\small $x_2$}};
\draw    (0.95,-1.1) node [magenta]{{\small $x_3$}};

\draw    (1.82,0.25) node [blue]{{\small $ \alpha_1$}};
\draw    (1.92,-0.25) node [blue]{{\small $ \alpha_3$}};

\draw    (1.23,0.8) node [blue]{{\small $ \alpha_1$}};
\draw    (0.58,1.08) node [blue]{{\small $ \alpha_2$}};

\draw    (-0.7,1.22) node [blue]{{\small $ \alpha_2$}};
\draw    (-1.2,1.05) node [blue]{{\small $ \alpha_3$}};

\draw    (-1.85,-0.2) node [blue]{{\small $ \alpha_1$}};
\draw    (-1.9,0.22) node [blue]{{\small $ \alpha_3$}};

\draw    (-1.2,-0.8) node [blue]{{\small $ \alpha_1$}};
\draw    (-0.58,-1.07) node [blue]{{\small $ \alpha_2$}};

\draw    (0.7,-1.21) node [blue]{{\small $ \alpha_2$}};
\draw    (1.25,-1.05) node [blue]{{\small $ \alpha_3$}};

\draw (a)--(b)--(c)--(d)--(e)--(f)--(a);

\draw [black] (2,0) --(2.1,-0.6);
\draw [black] (2,0) --(1.9,0.6);

\draw [black] (0.9,1) --(0.4,1.2);
\draw [black] (0.9,1) --(1.4,0.8);

\draw [black] (-1,1.2) --(-0.4,1.4);
\draw [black] (-1,1.2) --(-1.6,1);

\draw [black] (-2,0) --(-2.1,0.6);
\draw [black] (-2,0) --(-1.9,-0.6);

\draw [black] (-0.9,-1) --(-0.4,-1.2);
\draw [black] (-0.9,-1) --(-1.4,-0.8);

\draw [black] (1,-1.2) --(0.4,-1.4);
\draw [black] (1,-1.2) --(1.6,-1);

\end{tikzpicture}
    \caption{A central symmetric convex body $K$ with an inscribed hexagon.}
    \label{F-3}
\end{figure}

\subsection{Reduced Hamiltonian system}

First we show that  \eqref{E-3by3} can be reduced further to a $2\times2$ system of ODEs for two variables  $x_1(t)$ and $x_2(t)$.

\begin{Lemma}
    The internal angles of hexagon $x:=x_1(t)$ and $y:=x_2(t)$ satisfy the following system of ODEs:
    \begin{equation}\label{E-ODExy}
        \begin{cases}
            \dot{x}=\cos{(x+y)}-\cos{y},\\
            \dot{y}=\cos{x}-\cos{(x+y)}.
        \end{cases}
    \end{equation}
\end{Lemma}
\begin{proof}
 Recall that
    \begin{equation}\label{E-x123}
        x_1+x_2+x_3=2 \pi,
    \end{equation}
and we have the following relations between  $x_i$'s and $ \alpha_i$'s:
\begin{equation*}
    \begin{cases}
        x_1+ \alpha_1+ \alpha_3=\pi,\\
        x_2+ \alpha_2+ \alpha_1=\pi,\\
        x_3+ \alpha_3+ \alpha_2=\pi.
    \end{cases}
\end{equation*}
Solving this system and using \eqref{E-x123}, we get 
\begin{equation}\label{E-xalpha}
     \alpha_i=x_{i+2}-\frac{\pi}{2}, \qquad    i=1,2,3 .
\end{equation}
  Substituting the $ \alpha_i$'s into \eqref{E-3by3} and using \eqref{E-x123}, we obtain \eqref{E-ODExy}.
\end{proof}

We also note that, in the setting of Theorem~\ref{thm:Zindler_car}, the interior angles of the inscribed hexagon necessarily satisfy
\begin{equation}\label{range_of_x}
\frac{\pi}{2} < x_i(t) < \pi 
\qquad \text{for all } i.
\end{equation}
Indeed, by Theorem~\ref{thm:Zindler_car}, for each $t$ the vertices
$v_1(t),\dots,v_6(t)$ form a convex hexagon, and therefore 
$x_i(t)<\pi$ for all $i$. On the other hand, if $x_{i+2}(t)\le \frac{\pi}{2}$, then \eqref{E-xalpha} gives $\alpha_i(t)\le 0$, which contradicts the strict convexity of $K$.

\begin{Prop}\label{P-ODE}
The system of ODEs~\eqref{E-ODExy} with initial conditions $x(0)=x_0$ and $y(0)=y_0$ admits a unique solution. Moreover, the system is Hamiltonian with
\begin{equation}\label{E-Area}
H(x,y):=\sin x+\sin y-\sin(x+y)
\end{equation}
playing the role of the Hamiltonian. Namely,
\begin{equation}
\dot x=-\frac{\partial H}{\partial y},\qquad
\dot y=\frac{\partial H}{\partial x}.
\end{equation}
In particular, $H$ is constant along every orbit.
\end{Prop}
\begin{proof}
Since the right-hand side of \eqref{E-ODExy} is smooth, the system admits a
unique local solution for every initial condition. The Hamiltonian form follows immediately from \eqref{E-ODExy}:
\begin{equation}
-\frac{\partial H}{\partial y}=-\cos y+\cos(x+y)=\dot x,
\qquad
\frac{\partial H}{\partial x}=\cos x -\cos(x+y)=\dot y.
\end{equation}
\end{proof}

\begin{Rem}
One can show that
$H=\frac{S}{4}$, where $S$ denotes the area of the inscribed hexagon. The fact that this area of the Zindler's $N$-gon is preserved is classical and goes back to Auerbach~\cite{Auerbach1938}.
\end{Rem}

\begin{Prop}\label{P-boundsH}
Let
$H(x,y)=\sin x+\sin y-\sin(x+y)$,
and let
\begin{equation}
D:=\Bigl\{(x,y):\ \frac{\pi}{2}<x<\pi,\ \frac{\pi}{2}<y<\pi,\ x+y<\frac{3\pi}{2}\Bigr\}
\end{equation}
be the admissible region according to \eqref{range_of_x} and \eqref{E-x123}.
Then
\begin{equation}
H_0:=\max_{\partial D} H = 1+\sqrt2.
\end{equation}
Moreover, if a nonconstant closed orbit of the Hamiltonian system \eqref{E-ODExy} is contained in $D$, then  
\begin{equation}
H_0<H(x,y)<H_{\max}:=\frac{3\sqrt3}{2}.
\end{equation}
\end{Prop}

\begin{proof}
 
First we observe that $H$ is strictly concave on $D$. Indeed,
\begin{equation}
\nabla^2 H(x,y)=
\begin{pmatrix}
-\sin x+\sin(x+y) & \sin(x+y)\\
\sin(x+y) & -\sin y+\sin(x+y)
\end{pmatrix}.
\end{equation}
Since on $D$ one has
\begin{equation}
\sin x>0,\qquad \sin y>0,\qquad \sin(x+y)<0,
\end{equation}
it follows that 
\begin{equation}
\det \nabla^2H=
\sin x\sin y-\sin(x+y)\bigl(\sin x+\sin y\bigr)>0.
\end{equation}
Thus $\nabla^2H$ is negative definite on $D$, so $H$ is strictly concave. Moreover, the unique critical point in  $D$  
is
$
\left(\frac{2\pi}{3}, \frac{2\pi}{3}\right)
$, and thus   $
H_{\max} := H\left(\frac{2\pi}{3}, \frac{2\pi}{3}\right)=\frac{3\sqrt{3}}{2}
$.
\medskip

For the lower bound we first compute $H_0=\max_{\partial D}H$.
The boundary of $D$ has exactly three edges: $x=\frac{\pi}{2}$, $y=\frac{\pi}{2}$, and $x+y=\frac{3\pi}{2}$.
When $y=\frac{\pi}{2}$, we have
\begin{equation}
H\Bigl(x,\frac{\pi}{2}\Bigr)
=\sin x+1-\sin\Bigl(x+\frac{\pi}{2}\Bigr)
=1+\sin x-\cos x
=1-\sqrt2\cos\Bigl(x+\frac{\pi}{4}\Bigr)
\le 1+\sqrt2.
\end{equation}
By symmetry, the same bound holds on the edge $x=\frac{\pi}{2}$. When $x+y=\frac{3\pi}{2}$,  we have
\begin{equation}
H\Bigl(x,\frac{3\pi}{2}-x\Bigr)
=
\sin x+\sin\Bigl(\frac{3\pi}{2}-x\Bigr)-\sin\frac{3\pi}{2}
=
\sin x-\cos x+1
\le 1+\sqrt2.
\end{equation}
The equality is attained at $x=\frac{3\pi}{4}$, thus
\begin{equation}
H_0=\max_{\partial D}H=1+\sqrt2.
\end{equation}
Consequently, for every $c\in\mathbb R$, the superlevel set
\begin{equation}
\Omega_c:=\{(x,y)\in D:\ H(x,y)\ge c\}
\end{equation}
is convex (see~Figure~\ref{fig1} for visualization) and thus connected.
\begin{figure}
    \centering
    \includegraphics[width=0.5\linewidth]{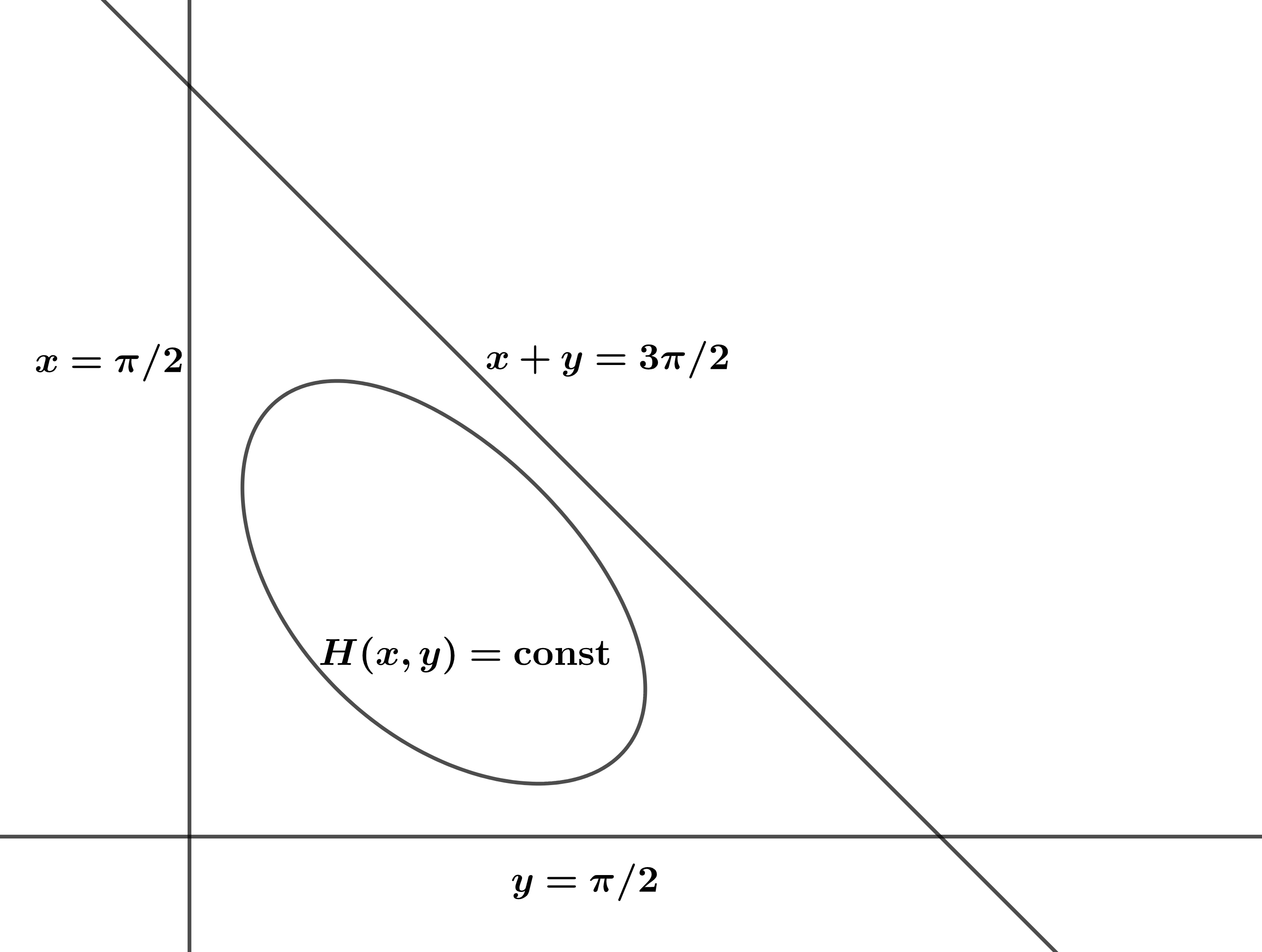}
    \caption{Level curve of $H$.}
    \label{fig1}
\end{figure}

Now let $\gamma\subset D$ be a nonconstant closed orbit, and let $H\equiv c$ on $\gamma$.
Since the only critical point of $H$ in $D$ is the strict maximum, a nonconstant closed orbit cannot reduce to that point, and therefore $c<H_{\max}$.

It remains to prove that $c>H_0$. Assume, for contradiction, that $c\le H_0$. Since
$H_0=\max_{\partial D}H$,
there exists a point $p\in\partial D$ such that $H(p)\ge c$. Thus the superlevel set $\Omega_c:=\{(x,y)\in D\colon H(x,y)\ge c\}$ meets $\partial D$.
On the other hand, the closed orbit $\gamma\subset\{H=c\}$ lies entirely inside the open domain $D$. By the Jordan curve theorem, $\gamma$ bounds a bounded planar region $U\subset D$. Since $H$ is constant on $\gamma$, and $H$ has only one interior critical point in $D$, namely the strict maximum, it follows that $H>c$ in $U$. Thus $U\subset \Omega_c$. We see that $\Omega_c$ contains points in $U$ and also points arbitrarily close to $\partial D$, thus outside $U$. Since $\Omega_c$ is convex, the line segment joining such a point in $U$ to such a point outside $U$ must lie entirely in $\Omega_c$. But any such segment must cross the boundary curve $\gamma$. This is impossible, because $\gamma\subset\{H=c\}$ is the boundary between the regions $\{H>c\}$ and $\{H<c\}$, so $\Omega_c=\{H\ge c\}\cap D$ cannot cross $\gamma$. The contradiction proves that
$c>H_0$.
\end{proof}

\subsection{One-dimensional reduction and period of $u(t)$}\label{subsec:1d_reduction}

Recall that along every solution $(x(t),y(t))$ of \eqref{E-ODExy}, the quantity
\begin{equation}\label{eq:H_level}
H =\sin x+\sin y-\sin(x+y)
\end{equation}
is conserved (see \eqref{E-Area}).
Introduce
\begin{equation} 
u:=\frac{x+y}{2} \qquad \text{and} \qquad v:=\frac{x-y}{2}.
\end{equation}
In these variables
\begin{equation}\label{eq:H_uv}
H=2\sin u\,(\cos v-\cos u)
\end{equation}
and
\begin{equation}\label{eq:cosv_formula}
\cos v=\cos u+\frac{H}{2\sin u}
=\frac{H+\sin(2u)}{2\sin u}.
\end{equation}
Summing the equations in \eqref{E-ODExy}, we get
\begin{equation}\label{eq:u_dot_from_xy}
\dot u=\frac{\dot x+\dot y}{2}=\frac{\cos x-\cos y}{2}
=\frac{\cos(u+v)-\cos(u-v)}{2}
=-\sin u\,\sin v.
\end{equation}
Using \eqref{eq:cosv_formula}, we have
\begin{equation}
\sin^2 v
=1-\cos^2 v
=1-\Bigl(\frac{H+\sin(2u)}{2\sin u}\Bigr)^2 =\frac{4\sin^2 u-(H+\sin(2u))^2}{4\sin^2 u}. \label{eq:sinv_sq}
\end{equation}
Therefore,
\begin{equation}\label{eq:scalar_reduction_QS}
\dot u^{\,2}
=\sin^2 u\,\sin^2 v
=\sin^2 u-\frac{(H+\sin 2u)^2}{4}
=:Q_H(u).
\end{equation}
Introducing
\begin{equation}\label{eqFGdef}
F(u):=2\sin u-\sin 2u \quad \text{and} \quad G(u):=2\sin u+\sin 2u,
\end{equation}
we rewrite 
\begin{equation}\label{eq:Q_factor_sec}
Q_H(u)=\frac{(F(u)-H)\,(G(u)+H)}{4}.
\end{equation}
On $(\pi/2,\pi)$ one has $\sin u>0$ and $1+\cos u>0$, thus $G(u)=2\sin u(1+\cos u)>0$, so $G(u)+H>0$ for all $H>0$.
Therefore, zeros of $Q_H$ in $(\pi/2,\pi)$ are exactly the solutions of $F(u)=H$. The following lemma provides necessary bounds for them when $H\in(H_0,H_{\max}) = (1+\sqrt2,\frac{3\sqrt3}{2})$.
\begin{Lemma}\label{LemUpm}
The function $F$ is increasing on $(\pi/2,2\pi/3)$ and is decreasing on $(2\pi/3,\pi)$. For each $H \in (H_0,H_{\max})$ the equation $F(u)=H$ has two simple solutions in $(\pi/2,3\pi/4)$,
\begin{equation}
u_-(H)\in\Bigl(\frac{\pi}{2},\frac{2\pi}{3}\Bigr)\quad \text{and} \quad
u_+(H)\in\Bigl(\frac{2\pi}{3},\frac{3\pi}4\Bigr),
\end{equation}
and $F(u)>H$ for $u \in (u_-(H),u_+(H))$. 
\end{Lemma}
\begin{proof}
We have $F'(u)=2\cos(u)-2\cos(2u)$, thus $F$ is increasing on $(\pi/2,2\pi/3)$ and is decreasing on $(2\pi/3,\pi)$. The maximum on $(\pi/2,\pi)$ is attained at $u_0=2\pi/3$ and is equal to $F(u_0)=H_{\max}$.
We also have $F(\pi/2)=2<H_0= F(3\pi/4)$. Together with monotonicity and continuity of $F$, this proves the claim.
\end{proof}
Since $Q_H$ is positive on $(u_-(H),u_+(H))$ and vanishes simply at the endpoints, the corresponding trajectory $u(t)$ is periodic and oscillates between the turning points $u_-(H)$ and $u_+(H)$. Along a monotone branch,
\begin{equation}
\dot u=\pm \sqrt{Q_H(u)},
\end{equation}
and therefore
\begin{equation}
dt=\frac{du}{\sqrt{Q_H(u)}}.
\end{equation}
It follows that the time needed to move from $u_-(H)$ to $u_+(H)$ is
\begin{equation}
\int_{u_-(H)}^{u_+(H)}\frac{du}{\sqrt{Q_H(u)}}.
\end{equation}
By time-reversal symmetry, the return from $u_+(H)$ to $u_-(H)$ takes the same time, so the full period is
\begin{equation}\label{eq:T_of_H}
T(H):=2\int_{u_-(H)}^{u_+(H)}\frac{du}{\sqrt{Q_H(u)}}.
\end{equation}

\begin{Prop}\label{prop:T_period}
For every $H\in(H_0,H_{\max})$ one has
\begin{equation}\label{eq:TS_two_sided_coarse}
\pi\sqrt{3/2}<\;T(H)\;<\;\frac{2\pi}{\sqrt{(4-\sqrt2)/2}}.
\end{equation}
\end{Prop}

\begin{proof}
Fix $H\in(H_0,H_{\max})$ and write $u_\pm=u_\pm(H)$.
We prove in Appendix~\ref{app:bounds} that 
\begin{equation}
\frac{4-\sqrt2}{2}\,(u-u_-)(u_+-u) \leq Q_H(u)\ \le\ \frac{8}{3}\,(u-u_-)(u_+-u)
\quad\text{for all }u\in[u_-,u_+].
\end{equation}
  Therefore, the statement of the proposition follows from the definition of $T(H)$ and the standard integral
\begin{equation}
 \int_{u_-}^{u_+}\frac{du}{\sqrt{(u-u_-)(u_+-u)}}
=  \pi. \qedhere
\end{equation}
\end{proof}

\subsection{Estimates for the radius vector}
We express the distance $r$ (the radius) from the origin $O$ to a vertex $R$ of the hexagon in terms of $x$ and $y$ (see Figure~\ref{F-Hexagon1}).  
\begin{figure}
    \centering
    \begin{tikzpicture}[scale=2.5]
\coordinate (a) at (2,0);
\coordinate (b) at (0.9,1);
\coordinate (c) at (-1,1.2);
\coordinate (d) at (-2,0);
\coordinate (e) at (-0.9,-1);
\coordinate (f) at (1,-1.2);

\draw (a)--(b)--(c)--(d)--(e)--(f)--(a);
\draw (b)--(f);
\draw (c)--(f);

\draw    (1.8,0) node [blue]{{\small $x$}};
\draw    (0.5,0.9) node [blue]{{\small $y+\frac{x}{2}-\frac{\pi}2$}};

\draw    (0.1,0.1) node [black]{{\small $O$}};

\draw    (1.45,0.7) node [black]{{\small $2$}};
\draw    (-0.5,0.5) node [black]{{\small $r$}};
\draw    (0.5,-0.5) node [black]{{\small $r$}};
\draw    (-1.45,-0.7) node [black]{{\small $2$}};
\draw    (0,1.2) node [black]{{\small $2$}};
\draw    (0,-1.2) node [black]{{\small $2$}};
\draw    (-1.7,0.7) node [black]{{\small $2$}};
\draw    (1.7,-0.7) node [black]{{\small $2$}};

\draw (0,0) node[circle,fill,inner sep=1.5pt] {};
\draw    (-0.96,1.26) node [black]{{\small $R$}};

\end{tikzpicture}
    \caption{A hexagon with side lengths equal to $2$.}
    \label{F-Hexagon1}
\end{figure}
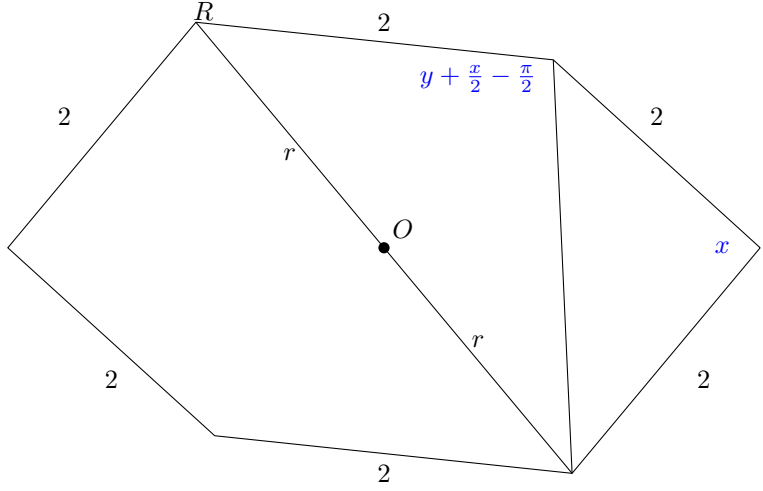
Using the properties of an isosceles triangle together with the Law of Cosines, we obtain
\begin{equation}
4r^2=4+16\sin^2{\frac{x}{2}} - 16\sin{\frac{x}{2}}\cos{\left(y+\frac{x}{2}-\frac{\pi}{2}\right)}.
\end{equation}
After simplifications using trigonometrical identities, we obtain 
\begin{equation} 
    r^2=3-2\cos{x}-2\cos{y}+2\cos{(x+y)}=1-4\cos{\frac{x+y}{2}}\Bigl(\cos{\frac{x-y}{2}}-\cos{\frac{x+y}{2}}\Bigr).
\end{equation}
Combining this with \eqref{eq:H_uv}, we rewrite it in terms of $u=\frac{x+y}{2}$ and $ v=\frac{x-y}{2}$: 
\begin{equation}\label{eq:r_u_v_system}
\begin{cases}
r^2 = 1 - 2H\cot u,\\ 
H = 2 \sin u \bigl(\cos v - \cos u\bigr).
\end{cases}
\end{equation}

\begin{Lemma}\label{Period_r} $T(H)$ is the minimal period of $r(t)$. \end{Lemma}
\begin{proof} Fix the $H\in(H_0,H_{\max})$.
Along the associated solution $(x(t),y(t))$ of~\eqref{E-ODExy}, the variable $u(t)=(x(t)+y(t))/2$ satisfies
\eqref{eq:scalar_reduction_QS} and has the minimal period $T(H)$ given by~\eqref{eq:T_of_H}.
Moreover, by \eqref{eq:r_u_v_system} on this orbit we have
\begin{equation}
r^2(t)=1-2H\cot u(t),\qquad u(t)\in\Bigl(\frac{\pi}{2},\pi\Bigr).
\end{equation}
Since $u\mapsto 1-2H\cot u$ is strictly monotone on $(\pi/2,\pi)$, the functions $u(t)$ and $r(t)$ have the same set of
periods; in particular, $r(t)$ has the minimal period $T(H)$.
\end{proof}

Using \eqref{eq:r_u_v_system}, we obtain the following estimate for the radius vector of $K$.

\begin{Lemma}\label{L-radius}
For every $H\in (H_0,H_{\max})$ the radius vector of $K$ satisfies
\begin{equation}
r(t)\le 1+\sqrt2\qquad\text{for all }t.
\end{equation}
\end{Lemma}

\begin{proof}
Since $(-\cot u)'=\csc^2u>0$ on $(\pi/2,\pi)$, the 
maximal value $r_{\max}$ of~$r(t)$ along a fixed-$H$ orbit is attained at the same moment as the maximal value of~$u(t)$, that is, when 
$u=u_+(H)< \frac{3\pi}{4}$, see Lemma~\ref{LemUpm}. Using the relation \eqref{eq:r_u_v_system} at $u=u_+(H)$, we obtain
\begin{equation}
r_{\max}^2
=1-2H\cot u_+(H)
=1-4\cos u_+(H)\,\bigl(1-\cos u_+(H)\bigr)
=\bigl(1-2\cos u_+(H)\bigr)^2.
\end{equation}
Since $u_+(H)\in(2\pi/3,3\pi/4)$, one has $\cos u_+(H)<0$, and thus
\begin{equation}
r_{\max}=1-2\cos u_+(H) < 1-2\cos(3\pi/4)=1+\sqrt{2}.\qedhere
\end{equation}
\end{proof}

\section{Proof of the main result}\label{sec:Main_th}

We prove the theorem by contradiction. According to Corollary~\ref{L-center}, we may assume that there exists a centrally symmetric convex body $K$, distinct from a disk, satisfying the hypotheses of Theorem~\ref{thm:main}.

Recall that $\gamma$ is an arc-length parametrization of the boundary of~$K$ (thus $|\gamma'(t)|=1$). 
Then the perimeter $P$ of the body $K$ is the \emph{geometric period} of~$\gamma$, i.e.
\begin{equation}
\gamma(t+P)=\gamma(t)\qquad\text{for all }t\in\mathbb R,
\end{equation}
and $P$ is the smallest positive number with this property.

Let $R_\theta$ denote the rotation by~$\theta$ about the center of symmetry (which we place at the origin).
Consider the set of rotational symmetries
\begin{equation}
\mathrm{Rot}(K):=\{\theta\in[0,2\pi):\ R_\theta(K)=K\}.
\end{equation}
If $K$ is not a disk, $\mathrm{Rot}(K)$ is a finite cyclic group; define the \emph{basic rotation angle}
\begin{equation}
\sigma_K:=\min\bigl(\mathrm{Rot}(K)\cap(0,2\pi)\bigr).
\end{equation}
Correspondingly, define $\eta_K\in(0,P)$ to be the unique number such that
\begin{equation}\label{eq:rotation_shift}
R_{\sigma_K}\bigl(\gamma(t)\bigr)=\gamma(t+\eta_K)\qquad\text{for all }t\in\mathbb R.
\end{equation}
Since $K$ is centrally symmetric, $R_\pi(K)=K$, thus $\pi\in\mathrm{Rot}(K)$ and therefore the order of $\mathrm{Rot}(K)$ is even, that it is equal to $2k$ for some $k\in\mathbb N$. Thus, we have   
\begin{equation}\label{E-sigma_k}
    \sigma_K=\frac{2\pi}{2k}=\frac{\pi}{k}
\qquad\text{and}\qquad
\eta_K=\frac{P}{2k}.
\end{equation}

We establish the perimeter bounds that follow from convexity. Since $K$ contains the inscribed convex hexagons of perimeter 12 and is contained in the disk of radius $1+\sqrt{2}$ (by~Lemma~\ref{L-radius}), we have
\begin{equation}\label{eq:P_bounds}
12\le P\le 2\pi(1+\sqrt2).
\end{equation}
Next, by definition \eqref{eq:rotation_shift}, we have $R_{\sigma_K}(\gamma(t))=\gamma(t+\eta_K)$,
and thus
\begin{equation}
r(t+\eta_K)=|\gamma(t+\eta_K)|
=|R_{\sigma_K}\gamma(t)|
=|\gamma(t)|
=r(t), \quad t \in \mathbb{R}.
\end{equation}
Therefore $\eta_K$ is a period of $r(t)$, and by Lemma~\ref{Period_r} there exists $m\in\N$ such that $\eta_K=m\,T(H)$.
Combining this with $\eta_K=\frac{P}{2k}$, we obtain
\begin{equation}\label{eq:T_quantized}
T(H)=\frac{P}{2km}.
\end{equation}
By Proposition~\ref{prop:T_period},
\begin{equation}\label{eq:T_bounds_main}
\pi\sqrt{\frac32}< T(H)< \frac{2\pi}{\sqrt{(4-\sqrt2)/2}}.
\end{equation}
Using \eqref{eq:P_bounds} and \eqref{eq:T_quantized}, we obtain
\begin{equation}
T(H)\le \frac{2\pi(1+\sqrt2)}{2km}
=\frac{\pi(1+\sqrt2)}{km}.
\end{equation}
Together with the lower bound in \eqref{eq:T_bounds_main}, this gives
\begin{equation}
\pi\sqrt{\frac32}< \frac{\pi(1+\sqrt2)}{km},
\end{equation}
so
\begin{equation}
km<\frac{1+\sqrt2}{\sqrt{3/2}}<2 \quad \quad \Rightarrow k=m=1.
\end{equation}
 Substituting this into \eqref{eq:T_quantized} and using \eqref{eq:P_bounds}, we obtain
\begin{equation}
T(H)=\frac{P}{2}\ge 6,
\end{equation}
which contradicts the upper bound from \eqref{eq:T_bounds_main}:
\begin{equation}
T(H)<\frac{2\pi}{\sqrt{(4-\sqrt2)/2}}<6.
\end{equation}
This contradiction shows that no such non-disk $K$ exists, so $K$ must be a disk.

\section{Remarks and open problems}\label{sec:rem}

A natural direction for future work is to investigate the analogous problem for perimeter densities $\sigma = \frac{1}{N}$ for larger $N$. The analysis becomes substantially more complicated for two related reasons. First, in contrast with the hexagonal case, one cannot in general express the auxiliary angles $\alpha_i(t)$ uniquely in terms of the variables $x_j(t)$. Second, even when such a parametrization is available, one should not expect the full system to collapse to a $2\times 2$ dynamical system. Thus the low-dimensional reduction that makes the present case tractable appears to be special to $\sigma = 1/6$.

Another interesting question concerns the period $T(H)$ of the corresponding Hamiltonian system. In the work of Bracho, Montejano, and Oliveros, a monotonicity of the period was observed numerically, and the same phenomenon appears in our setting as well. However, we do not currently have a rigorous proof of this monotonicity. Establishing such a result would be of independent interest and could provide a more conceptual route to rigidity.

Finally, although the reduced system in the present paper has a natural Hamiltonian structure, we have not exploited this structure in a truly effective way. It would be very interesting to understand whether the Hamiltonian viewpoint can be used more systematically, both to clarify the period function and to develop methods applicable to the higher $N$-gon cases.

\appendix
\section{Bounds on $Q_H(u)$} \label{app:bounds}
We start with a general lemma.

\begin{Lemma}[Parabolic domination]\label{lem:parabolic_domination}
Let $f\in C^2([a,b])$ satisfy $f(a)=f(b)=0$ and let $A\ge 0$.
\begin{itemize}
\item If $f''(x)\ge -2A$ for all $x\in[a,b]$, then $f(x)\le A(x-a)(b-x)$ for all $x\in[a,b]$.
\item If $f''(x)\le -2A$ for all $x\in[a,b]$, then $f(x)\ge A(x-a)(b-x)$ for all $x\in[a,b]$.
\end{itemize}
\end{Lemma}

\begin{proof}
Set $p(x):=A(x-a)(b-x)$, so $p(a)=p(b)=0$ and $p''(x)=-2A$.
Let $h:=f-p$. If $f''\ge -2A$, then $h''\ge 0$ so $h$ is convex; since $h(a)=h(b)=0$, convexity implies $h\le 0$,
i.e.\ $f\le p$. If $f''\le -2A$, then $h$ is concave and $h(a)=h(b)=0$ implies $h\ge 0$, i.e.\ $f\ge p$.
\end{proof}

Next, recall that
\begin{equation}\label{app:factorization}
Q_H(u)=\sin^2 u-\frac{(H+\sin 2u)^2}{4}=\frac{(F(u)-H)\,(G(u)+H)}{4},
\end{equation}
where
\begin{equation}
F(u)=2\sin u-\sin 2u,\qquad G(u)=2\sin u+\sin 2u,
\end{equation}
and 
\begin{equation} 
H\in(H_0,H_{\max}),\qquad
H_0=1+\sqrt2,\qquad
H_{\max}=\frac{3\sqrt3}{2}.
\end{equation}

 \begin{Lemma}\label{lem:Q_parabola_above}
Let $H\in(H_0,H_{\max})$ and let $u_\pm=u_\pm(H)$ be the turning points. Then
\begin{equation}\label{eq:Q_parabola_above}
Q_H(u)\ \le\ \frac{8}{3}\,(u-u_-)(u_+-u)
\qquad\text{for all }u\in[u_-,u_+].
\end{equation}
\end{Lemma}
\begin{proof}
Since $Q_H(u_\pm)=0$, by Lemma~\ref{lem:parabolic_domination} it suffices to prove
\begin{equation}
Q_H''(u) \ge -\frac{16}{3}
\qquad\text{for all }u\in[u_-,u_+].
\end{equation}
Differentiating $Q_H$ twice, we obtain
\begin{equation}
Q_H''(u)=2\Bigl(1+\cos 2u+H\sin 2u-2\cos^2(2u)\Bigr).
\end{equation}
From Lemma~\ref{LemUpm} for $u\in[u_-,u_+]$ we have $F(u)\geq H$, and also $u\in(\pi/2,3\pi/4)$, so
$\sin 2u\le 0$. This implies
$H\sin 2u\ge F(u)\sin 2u$,
and so
\begin{equation}
Q_H''(u)\ge 2\Bigl(1+\cos 2u+F(u)\sin 2u-2\cos^2(2u)\Bigr).
\end{equation}
We plug $F(u)=2\sin u-\sin 2u$, simplify, and rewrite the right-hand side as follows:
\begin{equation}
Q_H''(u)\ge 4\sin^2u\,(2\cos u+\cos 2u)=
-\frac{16}{3}+\frac43\Bigl((1+3\cos u)^2-6(1+\cos u)\cos^3u \Bigr).
\end{equation}
Since $\cos(u)\leq 0$, we obtain $Q_H''(u)\ge -\frac{16}{3}$.
\end{proof}

\begin{Lemma}\label{lem:Q_parabola_below}
Let $H\in(H_0,H_{\max})$ and let $u_\pm=u_\pm(H)$ be the turning points. Then
\begin{equation}\label{eq:Q_parabola_below}
Q_H(u)\ \ge\ \frac{4-\sqrt2}{2}\,(u-u_-)(u_+-u)
\qquad\text{for all }u\in[u_-,u_+].
\end{equation}
\end{Lemma}

\begin{proof}
Lemma~\ref{LemUpm} implies that 
 $[u_-,u_+]\subset(\pi/2,3\pi/4)$. Since
\begin{equation}
G'(u)=2\cos u+2\cos 2u<0
\qquad\text{on }\Bigl(\frac{\pi}{2},\frac{3\pi}{4}\Bigr),
\end{equation}
the function $G$ is strictly decreasing on $[u_-,u_+]$. Using $F(u_+)=H$, we get
\begin{equation}
G(u_+)+H
=
2\sin u_+ + \sin 2u_+ + (2\sin u_+ - \sin 2u_+)
=
4\sin u_+.
\end{equation} 
Therefore, for all $u\in[u_-,u_+]$,
\begin{equation}\label{eq:G_lower_bound_unified}
G(u)+H\ge G(u_+)+H=4\sin u_+\ge 4\sin\frac{3\pi}{4}=2\sqrt2.
\end{equation} 
To estimate $F(u)-H$ we write
\begin{equation}
F''(u)
=
-2\sin u+4\sin 2u
=
-2\sin u+4(2\sin u-F(u))
=
6\sin u-4F(u).
\end{equation}
We apply the estimate $F\geq H$ on $[u_-,u_+]$ (see Lemma~\ref{LemUpm}) and obtain
\begin{equation}
F''(u)\le 6-4H\le 6-4H_0=-2(2\sqrt2-1) \qquad\text{for all }u\in[u_-,u_+].
\end{equation}
Since $F(u_\pm)-H=0$, Lemma~\ref{lem:parabolic_domination} gives
\begin{equation}
F(u)-H\ge (2\sqrt2-1)(u-u_-)(u_+-u)
\qquad\text{for all }u\in[u_-,u_+].
\end{equation}
Combining this with \eqref{eq:G_lower_bound_unified}, we obtain
\begin{equation}
Q_H(u)
=
\frac{(F(u)-H)(G(u)+H)}{4}
\ge
\frac{4-\sqrt2}{2}\,(u-u_-)(u_+-u),
\end{equation}
which proves \eqref{eq:Q_parabola_below}.
\end{proof}

\bibliographystyle{plain}
\bibliography{Main}

@article {Zindler1921,
    AUTHOR = {Von Zindler, K.},
     TITLE = {{\"Uber} konvexe {G}ebilde},
   JOURNAL = {Monatsh. Math. Phys.},
  FJOURNAL = {Monatshefte f\"ur Mathematik und Physik},
    VOLUME = {31},
      YEAR = {1921},
    NUMBER = {1},
     PAGES = {25--56},
      ISSN = {1812-8076},
   MRCLASS = {99-04},
  MRNUMBER = {1549095},
       DOI = {10.1007/BF01702711},
       URL = {https://doi.org/10.1007/BF01702711},
}

@book {ScottishBook2015,
     TITLE = {The {S}cottish {B}ook},
    EDITOR = {Mauldin, R. D.},
   EDITION = {Second},
      NOTE = {Mathematics from the Scottish Caf\'e{} with selected problems
              from the new Scottish Book,
              Including selected papers presented at the Scottish Book
              Conference held at North Texas University, Denton, TX, May
              1979},
 PUBLISHER = {Birkh\"auser/Springer, Cham},
      YEAR = {2015},
     PAGES = {xvii+322},
      ISBN = {978-3-319-22896-9; 978-3-319-22897-6},
   MRCLASS = {00B25 (01A60 03E15)},
  MRNUMBER = {3242261},
MRREVIEWER = {Godofredo\ Iommi Amun\'ategui},
       DOI = {10.1007/978-3-319-22897-6},
       URL = {https://doi.org/10.1007/978-3-319-22897-6},
}

@article{Auerbach1938,
author = {Auerbach, H.},
journal = {Studia Mathematica},
language = {fre},
number = {1},
pages = {121-142},
title = {Sur un problème de M. Ulam concernant l'équilibre des corps flottants},
url = {http://eudml.org/doc/218616},
volume = {7},
year = {1938},
}

@article {Montejano1974,
    AUTHOR = {Montejano, L.},
     TITLE = {On a problem of {U}lam concerning a characterization of the
              sphere},
   JOURNAL = {Studies in Appl. Math.},
  FJOURNAL = {Studies in Applied Mathematics},
    VOLUME = {53},
      YEAR = {1974},
     PAGES = {243--248},
      ISSN = {0022-2526,1467-9590},
   MRCLASS = {52A10},
  MRNUMBER = {355822},
MRREVIEWER = {F.\ A.\ Valentine},
       DOI = {10.1002/sapm1974533243},
       URL = {https://doi.org/10.1002/sapm1974533243},
}

@article {Wegner2003,
    AUTHOR = {Wegner, F.},
     TITLE = {Floating bodies of equilibrium},
   JOURNAL = {Stud. Appl. Math.},
  FJOURNAL = {Studies in Applied Mathematics},
    VOLUME = {111},
      YEAR = {2003},
    NUMBER = {2},
     PAGES = {167--183},
      ISSN = {0022-2526,1467-9590},
   MRCLASS = {76B10 (76B75 76B99)},
  MRNUMBER = {1990237},
       DOI = {10.1111/1467-9590.t01-1-00231},
       URL = {https://doi.org/10.1111/1467-9590.t01-1-00231},
}

@article {BrachoMontejanoOliveros2001,
    AUTHOR = {Bracho, J. and Montejano, L. and Oliveros, D.},
     TITLE = {A Classification Theorem for {Z}indler {C}arrousels},
   JOURNAL = {Journal of Dynamical and Control Systems},
  FJOURNAL = {Journal of Dynamical and Control Systems},
    VOLUME = {7},
      YEAR = {2001},
    NUMBER = {3},
     PAGES = {367--384},
       DOI = {10.1023/A:1013099830164},
       URL = {https://doi.org/10.1023/A:1013099830164},
}

@article {BrachoMontejanoOliveros2004,
    AUTHOR = {Bracho, J. and Montejano, L. and Oliveros, D.},
     TITLE = {Carousels, {Z}indler curves and the floating body problem},
   JOURNAL = {Period. Math. Hungar.},
  FJOURNAL = {Periodica Mathematica Hungarica. Journal of the J\'anos Bolyai
              Mathematical Society},
    VOLUME = {49},
      YEAR = {2004},
    NUMBER = {2},
     PAGES = {9--23},
      ISSN = {0031-5303,1588-2829},
   MRCLASS = {52A10 (34A12 53A04)},
  MRNUMBER = {2106462},
       DOI = {10.1007/s10998-004-0519-6},
       URL = {https://doi.org/10.1007/s10998-004-0519-6},
}

@article{Wegner2007,
  title={Floating bodies of equilibrium in 2{D}, the tire track problem and electrons in a parabolic magnetic field},
  author={Wegner, F.},
  journal={arXiv preprint physics/0701241},
  year={2007}
}

@article {Falconer1983,
    AUTHOR = {Falconer, K. J.},
     TITLE = {Applications of a result on spherical integration to the
              theory of convex sets},
   JOURNAL = {Amer. Math. Monthly},
  FJOURNAL = {American Mathematical Monthly},
    VOLUME = {90},
      YEAR = {1983},
    NUMBER = {10},
     PAGES = {690--693},
      ISSN = {0002-9890,1930-0972},
   MRCLASS = {52A20 (43A90 52A22)},
  MRNUMBER = {723941},
MRREVIEWER = {V.\ P.\ Yashnikov},
       DOI = {10.2307/2323535},
       URL = {https://doi.org/10.2307/2323535},
}

@article{Schneider1970,
  title={Functional equations connected with rotations and their geometric applications},
  author={Schneider, R.},
  journal={Enseign. Math.(2)},
  volume={16},
  pages={297--305},
  year={1970}
}

@article {Ryabogin2023,
    AUTHOR = {Ryabogin, D.},
     TITLE = {On bodies floating in equilibrium in every orientation},
   JOURNAL = {Geom. Dedicata},
  FJOURNAL = {Geometriae Dedicata},
    VOLUME = {217},
      YEAR = {2023},
    NUMBER = {4},
     PAGES = {Paper No. 70, 17},
      ISSN = {0046-5755,1572-9168},
   MRCLASS = {52A38 (52A20)},
  MRNUMBER = {4601992},
MRREVIEWER = {Alina\ Stancu},
       DOI = {10.1007/s10711-023-00807-w},
       URL = {https://doi.org/10.1007/s10711-023-00807-w},
}

@article {Ryabogin2022,
    AUTHOR = {Ryabogin, D.},
     TITLE = {A negative answer to {U}lam's problem 19 from the {S}cottish
              {B}ook},
   JOURNAL = {Ann. of Math. (2)},
  FJOURNAL = {Annals of Mathematics. Second Series},
    VOLUME = {195},
      YEAR = {2022},
    NUMBER = {3},
     PAGES = {1111--1150},
      ISSN = {0003-486X,1939-8980},
   MRCLASS = {52A38 (52A20)},
  MRNUMBER = {4413748},
MRREVIEWER = {Andrea\ Colesanti},
       DOI = {10.4007/annals.2022.195.3.5},
       URL = {https://doi.org/10.4007/annals.2022.195.3.5},
}

@book {RyaboginSurvey,
     TITLE = {On flotation, stability and related questions: {A} survey},
    EDITOR = {M. Alfonseca and D. Ryabogin and A. Stancu and V. Yaskin},
   EDITION = {First},
 PUBLISHER = {Springer-Verlag, Berlin.},
      YEAR = {2026},
    VOLUME={32},
    series={to appear in New Probes into
Discrete and Convex Geometry (J. Pach and G. Toth, eds.)},
}

@book {Poussin1925,
     TITLE = {Leçons de mécanique analytique},
    EDITOR = {De La Vallée Poussin, C. J.},
 PUBLISHER = {Paris (in French)},
      YEAR = {1925},
    VOLUME={II},
}

@book {Zhukovsky1936,
     TITLE = {Classical mechanics},
    EDITOR = {Zhukovsky, N. E.},
 PUBLISHER = {Paris (in French)},
      YEAR = {1936},
}

@book {Dupin1822,
     TITLE = {Applications de géométrie et de méchanique à la marine, aux ponts-et-chaussées},
    EDITOR = {Dupin, C.},
 PUBLISHER = {Paris (in French)},
      YEAR = {1822},
}

\end{document}